\documentclass[12pt]{article}

\usepackage[margin=1.3in, top=1.3in, bottom=1.3in]{geometry}

\usepackage[utf8]{inputenc} 
\usepackage[T1]{fontenc}    
\usepackage{url}            
\usepackage{booktabs}       
\usepackage{amsfonts}       
\usepackage{nicefrac}       
\usepackage{microtype}      

\usepackage{amsmath,amssymb,amsthm,xcolor}
\usepackage{array}
\usepackage{float}

\theoremstyle{definition}
\newtheorem{example}{Example}

\newtheorem{definition}{Definition}

\theoremstyle{plain}
\newtheorem{theorem}{Theorem}

\newtheorem{proposition}{Proposition}

\usepackage{graphicx}
\usepackage{subcaption}
\usepackage{mdframed}
\usepackage{lipsum}
\newmdtheoremenv{theo}{Theorem}
\usepackage{wrapfig}
\usepackage[hidelinks,hypertexnames=false]{hyperref}
\usepackage{tikz}
\usepackage{enumitem}
\allowdisplaybreaks
\usepackage{pgfplots}
\usetikzlibrary{intersections, pgfplots.fillbetween}
\usepackage{epstopdf}
\usepackage{algorithmic}
\usepackage{booktabs}
\usepackage{changepage}

\title{\LARGE Sufficient conditions for instability of the subgradient method with constant step size}

\begin{document}

\author{\large C\'edric Josz\thanks{\url{cj2638@columbia.edu}, IEOR, Columbia University, New York. Research supported by NSF EPCN grant 2023032 and ONR grant N00014-21-1-2282.} \and Lexiao Lai\thanks{\url{ll3352@columbia.edu}, IEOR, Columbia University, New York.}}
\date{}

\maketitle

\begin{center}
    \textbf{Abstract}
    \end{center}
    \vspace*{-3mm}
 \begin{adjustwidth}{0.2in}{0.2in}
~~~~We provide sufficient conditions for instability of the subgradient method with constant step size around a local minimum of a locally Lipschitz semi-algebraic function. They are satisfied by several spurious local minima arising in robust principal component analysis and neural networks.
\end{adjustwidth} 
\vspace*{3mm}
\noindent{\bf Keywords:} Lyapunov stability, metric subregularity, subgradient method, Verdier condition

\section{Introduction}
\label{sec:Introduction}
The subgradient method with constant step size for minimizing a locally Lipschitz function $f:\mathbb{R}^n\rightarrow\mathbb{R}$ consists in choosing an initial point $x_0 \in \mathbb{R}^n$ and generating a sequence of iterates according to the update rule $x_{k+1} \in x_k - \alpha \partial f(x_k)$, for all $k \in \mathbb{N} := \{0,1,2,\hdots\}$, where $\alpha>0$ is the step size and $\partial f$ is the Clarke subdifferential \cite[Chapter 2]{clarke1990}. A notion of discrete Lyapunov stability \cite{josz2023lyapunov} was recently proposed to study the behavior of the subgradient method with constant size around a local minimum of a locally Lipschitz semi-algebraic function. Informally, a point is stable if all of the iterates of the subgradient method remain in any neighborhood of it, provided that the initial point is close enough to it and that the step size is small enough. 

It was shown that for a point to be stable, it is necessary for it to be a local minimum \cite[Theorem 1]{josz2023lyapunov} and it suffices for it to be a strict local minimum \cite[Theorem 2]{josz2023lyapunov}. If the function is additionally differentiable with a locally Lipschitz gradient, then it suffices to be a local minimum \cite[Proposition 3.3]{absil2005convergence}. In this note, we show that the existence of a Chetaev function \cite{chetaev1961stability} in a neighborhood of a non-strict local minimum satisfying certain geometric properties guarantees instability. Chetaev functions are similar to Lyapunov functions, except that they increase along the dynamics rather than decrease. We check that the geometric properties, which involve higher-order metric subregularity \cite{mordukhovich2015higher,zheng2015holder} and the Verdier condition \cite{verdier1976stratifications}, hold in several applications of interest and exhibit corresponding Chetaev functions. 

The Verdier condition was recently introduced to the field of optimization by Bianchi \textit{et al.} \cite{bianchi2022stochastic} and Davis \textit{et al.} \cite{davis2023active}. Those works extend to the nonsmooth setting the pioneering work by Pemantle \cite{pemantle1990nonconvergence} on the nonconvergence to strict saddle points of the perturbed gradient method with diminishing step size. Precisely, they consider the update rule $x_{k+1} \in x_k - \alpha_k (\partial f(x_k) + \epsilon_k)$ for all $k\in \mathbb{N}$, where there exist $0<c_1<c_2$ and $\gamma \in (1/2,1]$ such that $c_1/k^\gamma \leqslant \alpha_k \leqslant c_2/k^\gamma$ for all $k\in \mathbb{N}^* := \{1,2,3,\hdots\}$. Also, the random variable $\epsilon_k$ is drawn uniformly from a ball of radius $r>0$ centered at the origin. They prove nonconvergence to active strict saddles \cite[Definition 2.3]{davis2023active} satisfying the Verdier condition and an angle/proximal aiming condition \cite[Theorem 3]{bianchi2022stochastic} \cite[Theorem 6.2]{davis2023active}.

As shown by Lee \textit{et al.} \cite[Theorem 4]{lee2016} (see also \cite{panageas2017}), in the smooth setting and with constant step size, adding random noise is actually not necessary to prevent convergence to strict saddle points almost surely. More recently, it was observed \cite[Figure 3]{kleinberg2018alternative} that the gradient method with constant step size can escape spurious local minima after adding uniform random noise. A similar observation on the benefits of noise was made in \cite{keskar2016large} when training neural networks: large batch sizes tend to converge to sharp local minima \cite[Metric 2.1]{keskar2016large}, while small batch sizes tend to converge to flat local minima. Our work shows that critical points can be inherently unstable due to the local geometry of the objective function, without adding any noise.

\section{Sufficient conditions for instability}

Let $\|\cdot\|$ be the induced norm of an inner product $\langle \cdot, \cdot\rangle$ on $\mathbb{R}^n$. Let $B(a,r)$ and $\mathring{B}(a,r)$ respectively denote the closed ball and the open ball of center $a\in \mathbb{R}^n$ and radius $r>0$. We first recall the notion of discrete Lyupanov stability \cite[Definition 1]{josz2023lyapunov}. 

\begin{definition}
\label{def:discrete_lyapunov}
We say that $x^*\in \mathbb{R}^n$ is a stable point of a locally Lipschitz function $f:\mathbb{R}^n\rightarrow\mathbb{R}$ if for all $\epsilon>0$, there exist $\delta>0$ and $\bar{\alpha}>0$ such that for all $\alpha \in (0,\bar{\alpha}]$, the subgradient method with constant step size $\alpha$ initialized in $B(x^*,\delta)$ has all its iterates in $B(x^*,\epsilon)$.
\end{definition}

According to the above definition, a point $x^* \in \mathbb{R}^n$ is unstable if there exists $\epsilon>0$ such that for all $\delta>0$ and $\bar{\alpha}>0$, there exists $\alpha \in (0,\bar{\alpha}]$ and an initial point $x_0\in B(x^*,\delta)$ such that at least one of the iterates of the subgradient method with constant step size $\alpha$ does not belong to $B(x^*,\epsilon)$. The sufficient conditions proposed in this note actually imply instability in a stronger sense.

\begin{definition}
\label{def:strong_unstable}
We say that $x^*\in \mathbb{R}^n$ is a strongly unstable point of a locally Lipschitz function $f:\mathbb{R}^n\rightarrow\mathbb{R}$ if there exists $\epsilon>0$ such that for all but finitely many constant step sizes $\alpha>0$ and for almost every initial point in $B(x^*,\epsilon)$, at least one of the iterates of the subgradient method does not belong to $B(x^*,\epsilon)$.
\end{definition}

Recall that a point $x^* \in \mathbb{R}^n$ is a local minimum (respectively strict local minimum) of a function $f:\mathbb{R}^n\rightarrow\mathbb{R}$ if there exists a positive constant $\epsilon$ such that $f(x^*) \leqslant f(x)$ for all $x \in B(x^*,\epsilon)\setminus \{x^*\}$ (respectively $f(x^*) < f(x)$). A local minimum $x^* \in \mathbb{R}^n$ is spurious if $f(x^*)>\inf\{f(x):x\in \mathbb{R}^n\}$. In order to describe the nature of the set of critical points around a non-strict local minimum, we recall the definition of a smooth manifold.

\begin{definition}
\label{def:manifold}
A subset $S$ of $\mathbb{R}^n$ is a $C^p$ manifold with positive $p\in \mathbb{N}$ of dimension $m \in \mathbb{N}$ at $x \in S$ if there exists a Euclidean space $E$ of dimension $n-m$ such that there exists an open neighborhood $U$ of $x$ in $\mathbb{R}^n$ and a $p$ times continuously differentiable function $\varphi: U \rightarrow E$ such that $S\cap U = \varphi^{-1}(0)$ and whose Jacobian is surjective.
\end{definition}

We will use the following notions related to a $C^p$ manifold $S$ at point $x$. According to \cite[Example 6.8]{rockafellar2009variational}, the tangent cone $T_S(x)$ \cite[6.1 Definition]{rockafellar2009variational} and the normal cone $N_S(x)$ \cite[6.3 Definition]{rockafellar2009variational} at a point $x$ in $S$ are respectively the kernel of $\varphi'(x)$ and the range of $\varphi'(x)^*$ where $\varphi'(x)$ is the Jacobian of the function $\varphi$ in Definition \ref{def:manifold} at $x$ and $\varphi'(x)^*$ is its adjoint. 

In order to describe the variation of the objective function around a non-strict local minimum, we borrow the notion of metric $\theta$-subregularity of a set-valued mapping \cite{mordukhovich2015higher,zheng2015holder}. It is a generalization of metric subregularity \cite[Equation (4)]{van2015metric} \cite[Definition 2.3]{artacho2008characterization} \cite[Definition 3.1]{dontchev2004regularity} that has been used to study the Mordukhovich subdifferential \cite[Theorem 3.4]{mordukhovich2015higher}. Given $x\in \mathbb{R}^n$ and $S \subset \mathbb{R}^n$, let $d(x,S) := \inf \{\|x - y\| : y \in S\}$ and $P_S(x):= \mathrm{argmin} \{y \in S:\|x - y\|\}$. Also, given a set-valued mapping $F:\mathbb{R}^n\rightrightarrows\mathbb{R}^m$, let $\mathrm{graph}~F := \{ (x,y) \in \mathbb{R}^n \times \mathbb{R}^m : F(x) \ni y \}$.

\begin{definition}
\label{def:submetric}\cite[Definition 3.1 (i)]{mordukhovich2015higher} A mapping $F:\mathbb{R}^n\rightrightarrows\mathbb{R}^m$ is metrically $\theta$-subregular at $(\bar{x},\bar{y}) \in \mathrm{graph}~F$ with $\theta\in \mathbb{R}$ if there exist $c>0$ and a neighborhood $U$ of $\bar{x}$ such that $d(x,F^{-1}(\bar{y})) \leqslant c d(\bar{y},F(x))^\theta$ for all $x \in U$.
\end{definition}

We introduce two final definitions in order to further describe the variation of the objective function around a non-strict local minimum.
\begin{definition}
\label{def:riemannian gradient}
\cite[Definition 3.30]{boumal2023introduction}
Let $f:\mathbb{R}^n\rightarrow \mathbb{R}$ be a locally Lipschitz function and $S\subset \mathbb{R}^n$ be a $C^1$ manifold at $x$. We say that $f$ is $C^1$ on $S$ at $x$ if there exists a neighborhood $U$ of $x$ and a continuously differentiable function $\bar{f}:U \rightarrow \mathbb{R}$ such that $f(y) = \bar{f}(y)$ for all $y \in S\cap U$.
\end{definition}

According to \cite[Definition 3.58, Proposition 3.61]{boumal2023introduction}, the Riemannian gradient $\nabla_S f(x)$ of $f$ on $S$ at $x$ is given by $\nabla_S f(x):= P_{T_S(x)}(\nabla \bar{f}(x))$.

\begin{definition}
\label{def:verdier}
\cite[Definition 5 iii)]{bianchi2022stochastic} Let $f:\mathbb{R}^n\rightarrow \mathbb{R}$ be a locally Lipschitz function and let $S\subset \mathbb{R}^n$ be a $C^1$ manifold at a point $x^* \in \mathbb{R}^n$. Assume that $f$ is $C^1$ on $S$ at $x^*$. We say that $f$ satisfies the Verdier condition at $x^*$ along $S$ if there exist a neighborhood $U$ of $x^*$ and $c>0$ such that for all $y \in S\cap U$, $x \in U \setminus S$ and $v \in \partial f(x)$, we have $\left\|P_{T_S (y)}(v)-\nabla_{S} f(y)\right\| \leqslant c\|x-y\|$.
\end{definition}

The Verdier condition \cite[Equation (1.4)]{verdier1976stratifications} was introduced in 1976 to study the relationship between submanifolds arising in the Whitney stratification \cite{whitney1965tangents}. It was later shown that a finite family of definable sets always admits a Verdier stratification \cite[1.3 Theorem]{le1998verdier}, that is, for which the Verdier condition holds at every point on each stratum. Bianchi \textit{et al.} \cite{bianchi2022stochastic} and Davis \textit{et al.} \cite{davis2023active} recently used this condition to guarantee that a perturbed subgradient method on tilted functions with diminishing step size does not converge to active saddle points almost surely.

In the context of optimization, the Verdier condition poses a Lipschitz-like condition on the projection of the subgradients and the Riemannian gradient of the objective function along a $C^1$ manifold. Such a condition is reasonable since the domain of a continuous semi-algebraic function always admits a Verdier stratification such that the function satisfies the Verdier condition at every point along each stratum \cite[Theorem 1]{bianchi2022stochastic} \cite[Theorem 3.29]{davis2023active}. However, the manifold induced by the critical points around a non-strict local minimum may not be contained in any strata, in which case the Verdier condition need not hold. It is for this reason that the Verdier condition appears as an assumption in Theorem \ref{thm:suff_unstable} below. We illustrate the Verdier condition with the following two examples, where $\|\cdot\|$ is induced by the Euclidean inner product. They are illustrated in Figures \ref{fig:verdier} and \ref{fig:not_verdier} respectively.
\begin{example}
    Let $f:\mathbb{R}^2\rightarrow \mathbb{R}$ be the function defined by $f(x_1,x_2) := |x_1 x_2 - 1|$. It satisfies the Verdier condition at $x^*:= (1,1)$ along its set of critical points $S := \{(x_1,x_2) \in \mathbb{R}^2 :x_1 x_2 = 1\} \cup \{(0,0)\}$. Consider the neighborhood $U:= B(x^*,0.5)$ of $x^*$. For all $(y_1,y_2) \in S\cup U$, we have that $T_S(y_1,y_2) = \{(x_1,x_2)\in \mathbb{R}^2:y_2 x_1 + y_1 x_2 = 0\}$ and $\nabla_S f(y_1,y_2) = (0,0)$. For all $(x_1,x_2) \in U\setminus S$, we have that $\partial f(x_1,x_2) = \{(\mathrm{sign}(x_1x_2 - 1)x_2, \mathrm{sign}(x_1x_2 - 1)x_1)\}$, where $\mathrm{sign}(t) = 1$ if $t>0$ and $\mathrm{sign}(t) = -1$ if $t<0$. Thus, for all $(y_1,y_2) \in S\cap U$, $(x_1,x_2) \in U \setminus S$ and $v \in \partial f(x_1,x_2)$, we have
 \begin{align*}
     \|P_{T_S(y_1,y_2)}(v) - \nabla_S f(y_1,y_2)\| &=\frac{|x_1 y_2 - x_2 y_1|}{\sqrt{y_1^2 + y_2^2}}\\
     &= \frac{|(x_1 - y_1)y_2 - (x_2 - y_2) y_1|}{\sqrt{y_1^2 + y_2^2}} \\
     &\leqslant \sqrt{(x_1 - y_1)^2 + (x_2 - y_2)^2}\\
     &= \|(x_1,x_2) - (y_1,y_2)\|
 \end{align*}
by the Cauchy-Schwarz inequality.
\end{example}
\begin{example}
Let $f:\mathbb{R}^2\rightarrow \mathbb{R}$ be the function defined by $f(x_1,x_2) := \max\{- x_1^2+2x_2 , |x_2|\}$, which is slight modification of \cite[Example 3.1]{davis2023active}. It does not satisfy the Verdier condition at $x^*:= (0,0)$ along its set of critical points $S := \mathbb{R}\times \{0\}$. Indeed, consider the sequences $y^k := (1/k,0) \in S$, $x^k := (1/k,1/k^2) \notin S$ and $v^k:= (-2/k,2)$ defined for all $k \in \mathbb{N}^*$. They satisfy $y^k \rightarrow x^*$, $x^k \rightarrow x^*$, and $v^k \in \partial f(x^k)$, yet
\begin{equation*}
    \frac{\|P_{T_S(y^k)}(v^k) - \nabla_S f(y^k)\|}{\|x^k - y^k\|} = \frac{\|(-2/k,0) - (0,0)\|}{\|(1/k,1/k^2) - (1/k,0)\|} = \frac{2/k}{1/k^2} \rightarrow \infty.
\end{equation*}
\end{example}

\begin{figure}[ht]
\centering
\begin{subfigure}{.49\textwidth}
  \centering
  \includegraphics[width=1\textwidth]{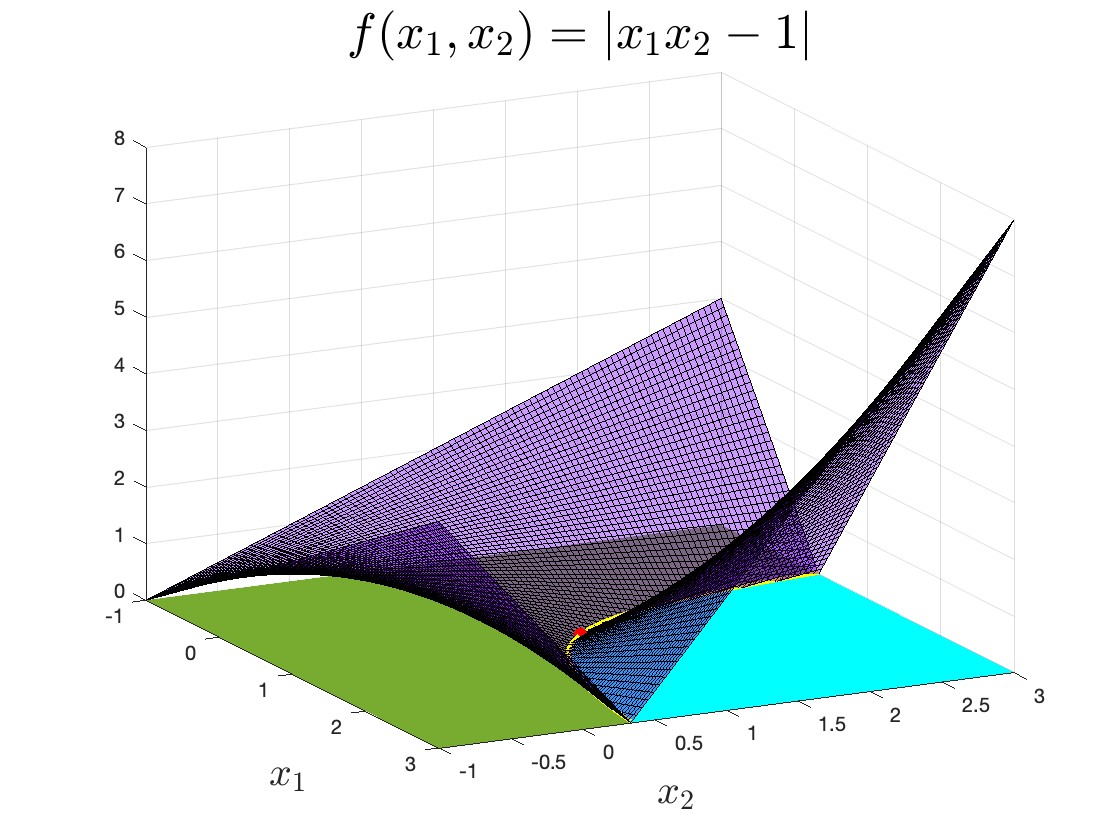}
  \caption{Verdier condition verified at $(1,1)$ along manifold of critical points.}
  \label{fig:verdier}
\end{subfigure}
\begin{subfigure}{.49\textwidth}
  \centering
  \includegraphics[width=1\textwidth]{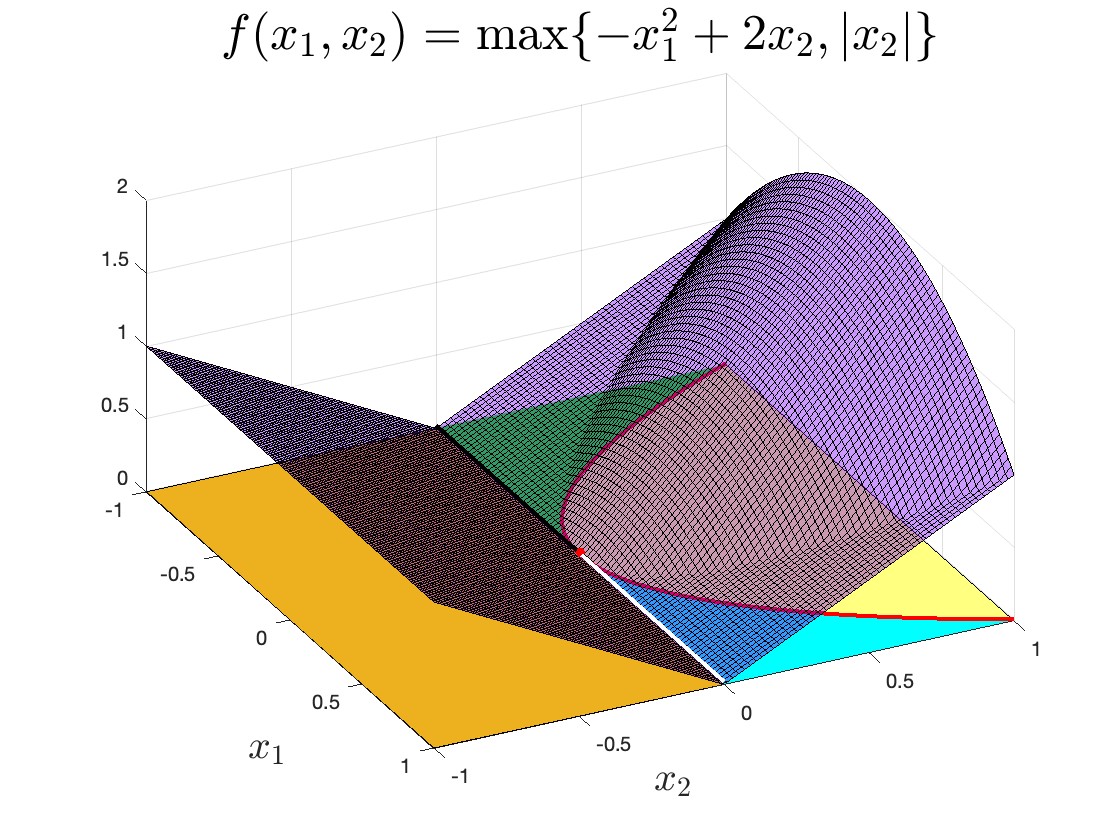}
  \caption{Verdier condition violated at $(0,0)$ along manifold of critical points.}
    \label{fig:not_verdier}
\end{subfigure}
\caption{Verdier stratification of the domain of two continuous semi-algebraic functions.}
\end{figure}

We are now ready to state our main result.

\begin{theorem}
\label{thm:suff_unstable}
Let $f:\mathbb{R}^n\rightarrow \mathbb{R}$ be a locally Lipschitz semi-algebraic function whose set of critical points we denote by $S$. Assume that $S$ is a $C^2$ manifold at some $x^*\in S$ of dimension less than $n$. Assume that there exist $\theta_1 \geqslant 0$, a neighborhood $U$ of $x^*$, and a continuous function $C:\mathbb{R}^n \rightarrow \mathbb{R}$ such that for all $\alpha>0$, there exist $c_1>0$ such that for any sequence $x_0,x_1,\hdots \in U \setminus S$ generated by the subgradient method with constant step size $\alpha$, we have $C(x_{k+1}) -  C(x_k) \geqslant c_1 d(x_k,S)^{\theta_1}$ for all $k\in \mathbb{N}$. The point $x^*$ is strongly unstable if 1) $\theta_1 = 0$ or 2) $\partial f$ is metrically $\theta_2$-subregular at $(x^*,0)$ with $\theta_2>1$ and $f$ satisfies the Verdier condition at $x^*$ along $S$.
\end{theorem}

\begin{proof} We begin with an outline of the proof. In order to establish instability, we reason by contradiction and assume that the iterates of the subgradient method remain in a neighborhood of a fixed critical point. We show that this implies that the function $C$ becomes unbounded along the iterates, which is impossible since this function is continuous. The key to showing unboundedness is to prove divergence of a series whose terms depend on the distance of the iterates to the manifold of critical points. For the proof to work, this distance should be positive for all iterates. We hence begin the proof by ensuring that this holds almost surely. After treating an easy case, the majority of the proof is devoted to showing that the distance to the set of critical points does not converge to zero.

We seek to show that there exists $\epsilon>0$ such that for all but finitely many constant step sizes $\alpha>0$, there exists a null subset $I_\alpha \subset \mathbb{R}^n$ such that for every initial point $x_0 \in B(x^*,\epsilon) \setminus I_\alpha$, at least one of the iterates of the subgradient method does not belong to $B(x^*,\epsilon)$. Since $S$ is a $C^2$ manifold at $x^*$ of dimension less than $n$, we have that $S\cap U$ is a semi-algebraic null set after possibly reducing $U$. 
By the cell decomposition theorem \cite[(2.11) p. 52]{van1998tame} and \cite[Claim 3]{bolte2020mathematical}, there exist $\alpha_1,\hdots,\alpha_m>0$ such that for all constant step sizes $\alpha \in (0,\infty) \setminus \{\alpha_1,\hdots,a_m\}$, there exists a null subset $I_\alpha \subset \mathbb{R}^n$ such that, for every initial point $x_0 \in \mathbb{R}^n \setminus I_\alpha$, none of the iterates $x_0,x_1,x_2,\hdots$ of the subgradient method belong to the semi-algebraic null set $S\cap U$.

\underline{Case 1: Assume that $\theta_1 = 0$.} Let $\epsilon >0$ such that $B(x^*,\epsilon) \subset U$. Let $\alpha \in (0,\infty) \setminus \{\alpha_1,\hdots,a_m\}$ and consider a sequence of iterates $x_0,x_1,x_2,\hdots \in \mathbb{R}^n$ of the subgradient method with constant step size $\alpha$ such that $x_0 \in B(x^*,\epsilon) \setminus I_\alpha$. We reason by contradiction and assume that $x_k \in B(x^*,\epsilon)$ for all $k \in \mathbb{N}$. Thus $x_k \notin S$ for all $k \in \mathbb{N}$. We have $C(x_{k+1}) - C(x_k) \geqslant c_1 d(x_{k},S)^{\theta_1}$ and
\begin{equation}
\label{eq:diverge}
    C(x_K) - C(x_0) = \sum\limits_{k = 0}^{K-1} C(x_{k+1}) - C(x_k) \geqslant \sum\limits_{k = 0}^{K-1} c_1 d(x_{k},S)^{\theta_1} = \sum\limits_{k = 0}^{K-1} c_1,
\end{equation}
which converges to $+\infty$ as $K$ converges to $+\infty$. Since $C$ is continuous and $x_K \in B(x^*,\epsilon)$, this yields a contradiction.

\underline{Case 2: Assume that $\theta_1>0$.} We proceed in four steps. We begin by choosing $\epsilon>0$ sufficiently small so that the objective function admits favorable geometric properties in $B(x^*,2\epsilon)$ (step 1). We then use these properties, including metric $\theta_2$-subregularity, to show that $d(x_{k+1},S) \geqslant d(x_k,S)$ whenever $d(x_k,S)$ is small enough (step 2). This prevents $d(x_k,S)$ from converging to zero. Similar to \eqref{eq:diverge}, this leads to a divergent series $\sum_{k = 0}^{\infty} c_1 d(x_{k},S)^{\theta_1}$ and hence to a contradiction. A computation reveals that proving the inequality on the distances reduces to showing that a certain ratio is bounded (step 3), at which point we invoke the Verdier condition. This in turn requires showing that the projection is preserved when taking a step of a slight modification of the subgradient method (step 4). 

\underline{Step 1} We begin by choosing $\epsilon>0$ such that the projection $P_S$ onto $S$ is Lipschitz continuous and identifies on $B(x^*,2\epsilon)$ with the preimage of a mapping related to the normal cone $N_S(x)$, among other properties.

Since $S$ is a $C^2$ manifold at $x^*$, $S\cap U$ is strongly amenable \cite[10.23 Definition (b)]{rockafellar2009variational} after possibly reducing $U$. It follows that $S\cap U$ is prox-regular \cite[13.31 Exercise, 13.32 Proposition]{rockafellar2009variational} and locally closed \cite[p. 28]{rockafellar2009variational}. Therefore, there exists a closed neighborhood $V \subset U$ of $x^*$ such that $S\cap V$ is closed and prox-regular at $x^*$. By \cite[Theorem 1.3 (j)]{poliquin2000local}, there exists $\epsilon>0$ such that the projection $P_{S\cap V}$ onto $S\cap V$ is single-valued and Lipschitz continuous with some constant $L>0$ on $B(x^*,2\epsilon)$. After possibly reducing $\epsilon>0$, we have $P_{S\cap V}(x) = P_S(x)$ for all $x\in B(x^*,2\epsilon)$. (Indeed, if $B(x^*,5\epsilon) \subset V$, then $\|x-y\|\geqslant 3\epsilon$ for all $y \in S \setminus V$ while $\|x-x^*\|\leqslant 2\epsilon$.) Again by \cite[Theorem 1.3 (j)]{poliquin2000local}, there exists $c>0$ such that $P_S(x) = (I+N_S^c)^{-1}(x)$ for all $x \in B(x^*,2\epsilon)$, where $N_S^c$ is a set-valued mapping defined from $\mathbb{R}^n$ to the subsets of $\mathbb{R}^n$ by
\begin{equation}
\label{eq:N_S^c}
     N_S^c(x) := \left\{\begin{array}{ll}
         N_S(x) \cap \mathring{B}(0,c) & \text{if $x\in S$}, \\
         \emptyset & \text{else}.
    \end{array}\right.
\end{equation}
After possibly reducing $\epsilon$, we may assume that $(2+L)\epsilon < c$. 

In the following, we further reduce $\epsilon$ whenever necessary.
Since $\partial f$ is metrically $\theta_2$-subregular at $(x^*,0)$, there exists $c_2>0$ such that $d(x,S) \leqslant c_2 d(0,\partial f(x))^{\theta_2}$ for all $x \in B(x^*,\epsilon)$. Since $f$ satisfies the Verdier condition at $x^*$ along $S$, there exists $c_3>0$ such that for all $y \in B(x^*,2\epsilon)\cap S$, $x \in B(x^*,2\epsilon) \setminus S$ and $v \in \partial f(x)$, we have $\left\|P_{T_{
S}(y)}(v)\right\| \leqslant c_3\|x-y\|$. Indeed, $\nabla_S f(y) = 0$ for all $y\in B(x^*,2\epsilon) \cap S$ because $f$ agrees with a constant function along $S$ around $x^*$ by the semi-algebraic Morse-Sard theorem \cite[Corollary 9]{bolte2007clarke}.

\underline{Step 2} Having chosen $\epsilon>0$, let $\alpha \in (0,\infty) \setminus \{\alpha_1,\hdots,a_m\}$. Consider a sequence of iterates $x_0,x_1,x_2,\hdots$ of the subgradient method with constant step size $\alpha$ such that $x_0 \in B(x^*,\epsilon) \setminus I_\alpha$. As in the case when $\theta_1 = 0$, we reason by contradiction and assume that $x_k \in B(x^*,\epsilon)$ for all $k \in \mathbb{N}$. Thus $x_k \notin S$ for all $k \in \mathbb{N}$. Also, for all $K \in \mathbb{N}$, we have $C(x_K) - C(x_0) \geqslant \sum_{k=0}^{K-1} c_1 d(x_{k},S)^{\theta_1}$. In order to show that $\sum_{k=0}^{K-1} c_1 d(x_{k},S)^{\theta_1}$ diverges, it suffices to show that $d(x_k,S)$ does not converge to zero. To this end, we next show that $d(x_{k+1},S) \geqslant d(x_k,S)$ whenever $d(x_k,S)>0$ is sufficiently small (it is non-zero because $x_k \notin S$).

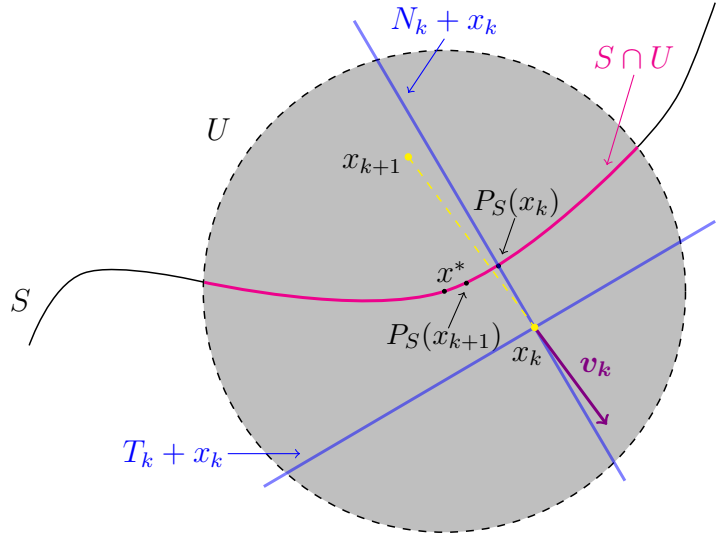
\begin{figure}[ht]
\centering
    \begin{tikzpicture}[scale=2.4]
\fill[gray!50] (5.5,.9) circle (38pt); 
    \draw[line width=.2mm,dashed] (5.5,.9) circle (38pt);
    \fill (3.15,0.85) node{\normalsize $S$};
    \draw[line width=.2mm] plot [smooth,tension=0.689] coordinates {(3.2,0.6)(3.5,1.0)(4.17,0.95)};
    \draw[magenta,line width=.4mm, name path=A] plot [smooth,tension=0.689] coordinates {(4.17,0.95) (5.5,.9) (6.57,1.7)};
    \draw[line width=.2mm] plot [smooth,tension=0.689] coordinates {(6.57,1.7)(6.8, 2)(7,2.5)};
    \draw[magenta,->] (6.55,2.1)-- (6.4,1.6);
    \fill (6.55,2.2)  node {\normalsize $\textcolor{magenta}{S\cap U}$};

    \fill (4.25,1.8)  node {\normalsize $U$};
    \filldraw (5.5,.9) circle (.3pt);
    \fill (5.54,1.01)  node {\normalsize $x^*$};

     \filldraw (5.622,0.945) circle (.3pt);
    \fill (5.5,0.65)  node {\small $P_S(x_{k+1})$};
     \draw[->](5.522,0.695)--(5.61,0.895);
     \filldraw (5.8,1.04) circle (.3pt);
     \fill (5.9,1.39)  node {\small $P_S(x_k)$};
     \draw[->](5.89,1.305)--(5.82,1.105);
     \draw[blue,opacity=0.5,line width=.4mm] plot [smooth,tension=1] coordinates {(5,2.4)(6.5,-0.15)};
     \fill (5.5,2.4)  node {\normalsize $\textcolor{blue}{N_k + x_k}$};
     \draw[blue,->] (5.5,2.3)-- (5.3,2);
     \draw[blue,opacity=0.5,line width=.4mm] plot [smooth,tension=1] coordinates {(4.5,-0.1824)(7,1.2882)};
     \fill (4,0) node {\normalsize $\textcolor{blue}{T_k + x_k}$};
     \draw[blue,->] (4.3,0)-- (4.7,0);
       \draw[line width=.4mm,violet,->] (6,.7)-- (6.4,.16);
       \fill (6.34,.48)  node {\normalsize $\textcolor{violet}{\boldsymbol{v_k}}$}; 
       \draw[yellow,dashed,line width=.2mm] plot [smooth, tension=1] coordinates {(5.3,1.6450)(6,.7)};
       
    \filldraw[yellow] (6,.7) circle (.5pt);
      \fill (5.96,.54)  node {\normalsize $x_k$};
       \filldraw[yellow] (5.3,1.6450) circle (.5pt);
       \fill (5.1,1.6)  node {\normalsize $x_{k+1}$};
    \end{tikzpicture}
\caption{Illustration of $d(x_{k+1},S) \geqslant d(x_k,S)$ for $d(x_k,S)$ sufficiently small.}
\label{fig:chetaev}
\end{figure}

Since $(x_k)_{k\in \mathbb{N}}$ is generated by the subgradient method with constant step size $\alpha$, for all $k \in \mathbb{N}$ there exists $v_k \in \partial f(x_k)$ such that $x_{k+1} = x_k - \alpha v_k$. 
As illustrated in Figure \ref{fig:chetaev}, we have
\begin{subequations}
    \begin{align}
        d(x_{k+1},S) & = \|x_{k+1} - P_S(x_{k+1})\| \label{dist_a} \\
        & = \|x_k - P_S(x_{k+1}) - \alpha v_k\| \label{dist_b} \\
        & \geqslant \alpha \|v_k\| - \|x_k-P_S(x_{k+1})\| \label{dist_c} \\
        & \geqslant \alpha c_2^{-1/\theta_2} d(x_k,S)^{1/\theta_2}- \|x_k-P_S(x_{k+1})\| \label{dist_d} \\
        & = d(x_k,S)\left(\alpha c_2^{-1/\theta_2} d(x_k,S)^{1/\theta_2-1} - \frac{\|x_k-P_S(x_{k+1})\|}{d(x_k,S)}\right) \label{dist_e}\\
        & \geqslant d(x_k,S) \label{dist_f}
    \end{align}
\end{subequations}
provided that $d(x_k,S)$ sufficiently small and that $\|x_k-P_S(x_{k+1})\|/d(x_k,S)$ is upper bounded on $B(x^*,\epsilon) \setminus S$ if $d(x_k,S)$ sufficiently small. Indeed, in \eqref{dist_a} $P_S(x_{k+1})$ is a singleton because $x_{k+1} \in B(x^*,\epsilon)$. \eqref{dist_b}-\eqref{dist_c} are deduced from the update rule and the triangular inequality. \eqref{dist_d} follows from the metric $\theta_2$-subregularity of $\partial f$ at $x^*$ for $0$. In the second factor of \eqref{dist_e}, the first term $c_2^{-1/\theta_2} d(x_k,S)^{1/\theta_2-1}$ diverges as $d(x_k,S)$ nears zero because $1/\theta_2-1< 0$. Hence, if the second term $\|x_k-P_S(x_{k+1})\|/d(x_k,S)$ is bounded, then the lower bound \eqref{dist_f} holds. 

\underline{Step 3} We next focus on proving that $\|x_k-P_S(x_{k+1})\|/d(x_k,S)$ is bounded. Since $x_k \in B(x^*, \epsilon)\setminus S$, $P_S(x_k) \in B(x^*,2\epsilon)\cap S$, and $v_k \in \partial f(x_k)$, by the Verdier condition we have $\|P_{T_S(P_S(x_k))}(v_k)\| \leqslant c_3 \|x_k - P_S(x_k)\|$ for all $k\in \mathbb{N}$. For notational convenience, let $T_k:= T_S(P_{S}(x_k))$ and $N_k:= N_S(P_{S}(x_k))$ respectively be the tangent and normal cones of $S$ at $P_S(x_k)$. Also, let $v^T_k := P_{T_k}(v_k)$ and $v^N_k := P_{N_k}(v_k)$ respectively be the projections of $v_k$ on $T_k$ and $N_k$. With these notations, we have $\|v_k^T\| \leqslant c_3\|x_k - P_{S}(x_k)\|$ for all $k\in \mathbb{N}$. Observe that
\begin{subequations}
    \begin{align}
        \frac{\|x_k-P_S(x_{k+1})\|}{d(x_k,S)} & \leqslant \frac{\|x_k-P_S(x_k)\| + \|P_S(x_k)-P_S(x_{k+1})\|}{d(x_k,S)} \label{proj_a} \\
        & = 1 + \frac{\|P_S(x_k)-P_S(x_{k+1})\|}{\|x_k-P_S(x_k)\|} \label{proj_b} \\
        & = 1 + \frac{\|P_S(x_k-\alpha v_k^N)-P_S(x_k-\alpha v_k)\|}{\|x_k-P_S(x_k)\|} \label{proj_c} \\
        & \leqslant 1 + L \alpha \frac{\|v_k^N-v_k\|}{\|x_k-P_S(x_k)\|} \label{proj_d} \\
        & \leqslant 1 + L \alpha c_3 \label{proj_e}
    \end{align}
\end{subequations}
provided that $d(x_k,S)$ sufficiently small. Indeed, \eqref{proj_a} follows from the triangular inequality. \eqref{proj_b} holds because $d(x_k,S) = \|x_k-P_S(x_k)\|$. \eqref{proj_c} holds because of the update rule $x_{k+1} = x_k-\alpha v_k$ and the fact that $P_S(x_k) = P_S(x_k-\alpha v_k^N)$, which is the object of the next step. \eqref{proj_d} holds because $P_S$ is $L$-Lipschitz continuous in $B(x^*,2\epsilon)$. Finally, \eqref{proj_e} follows from the Verdier condition and the fact that $v_k = v_k^T + v_k^N$.

\underline{Step 4} It remains to prove that $P_S(x_k) = P_S(x_k-\alpha v_k^N)$ when $d(x_k,S)$ is sufficiently small. We may thus assume that $d(x_k,S) \leqslant \epsilon/(\alpha c_3)$, which guarantees that $x_k-\alpha v_k^N \in B(x^*,2\epsilon)$. Indeed, 
\begin{subequations}
    \begin{align}
        \| x_k-\alpha v_k^N - x^*\| & \leqslant \| x_k - \alpha v_k - x^*\| + \alpha \| v_k^T\| \\
        & \leqslant \epsilon + \alpha c_3 \|x_k - P_S(x_k)\| \\
        & \leqslant \epsilon + \alpha c_3 d(x_k,S) \\
        & \leqslant 2\epsilon.
    \end{align}
\end{subequations}
Recall that $P_S(x) = (I+N_S^c)^{-1}(x)$ for all $x \in B(x^*,2\epsilon)$. We have
\begin{subequations}
    \begin{align}
        P_S( x_k-\alpha v_k^N ) & = (I+N_S^c)^{-1}(x_k-\alpha v_k^N) \label{normal_a} \\
        & = (I+N_S^c)^{-1}(P_S(x_k) + x_k - P_S(x_k) - \alpha v_k^N) \label{normal_b} \\
        & = P_S(x_k). \label{normal_c}
    \end{align}
\end{subequations}
Indeed, \eqref{normal_c} is equivalent to $P_S(x_k) + x_k - P_S(x_k) - \alpha v_k^N \in (I+N_S^c)(P_S(x_k))$, that is to say, $x_k - P_S(x_k) - \alpha v_k^N \in N_S^c(P_S(x_k))$. Since $P_S(x_k) \in S$, by definition of $N_S^c$ in \eqref{eq:N_S^c}, $N_S^c(P_S(x_k)) = N_k \cap \mathring{B}(0,c)$. To see why $x_k - P_S(x_k) - \alpha v_k^N \in N_k$, observe that $P_S(x_k) = P_S(x_k - P(x_k) + P(x_k)) = (I+N_S^c)^{-1}(x_k - P(x_k) + P(x_k))$. Thus $x_k - P(x_k) + P(x_k) \in (I+N_S^c)(P_S(x_k))  = P_S(x_k) + N_S^c(P_S(x_k))$, that is to say, $x_k - P(x_k) \in N_k$. Since $N_k$ is a linear subspace, it follows that $x_k - P_S(x_k) - \alpha v_k^N \in N_k$. Finally, 
    \begin{align*}
        \|x_k - P_S(x_k)-\alpha v_k^N\| & \leqslant \|x_k -\alpha v_k^N - x^*\| + \|x^* - P_S(x_k)\|\\
        & \leqslant 2\epsilon +  \|P_S(x^*) - P_S(x_k)\|\\
        & \leqslant 2\epsilon + L\|x^* - x_k\|\\
        & \leqslant (2 + L) \epsilon<c.
    \end{align*}
\end{proof}

In the next section, Theorem \ref{thm:suff_unstable} will be used to prove instability of spurious local minima in two practical problems; see Propositions \ref{prop:neural_network_unstable} and \ref{prop:rpca_rankr_unstable}. Recall that Lyapunov functions are used in the theory of ordinary differential equations (or inclusions) to prove the stability of an equilibrium point \cite{liapounoff1907probleme}. For example, a locally Lipschitz semi-algebraic objective function $f$ is a Lyapunov function for the continuous-time subgradient dynamics $x'\in -\partial f(x)$ around a strict local minimum $x^*$ \cite[Theorem 5.16 1.]{sastry2013nonlinear}. Indeed, $f$ is positive around $x^*$ and $f\circ x$ is decreasing along any trajectory $x$. The objective function is however not monotonic along discrete-time dynamics, in which case it ceases to be a Lyapunov function.

In contrast to Lyapunov functions, Chetaev functions are used to prove instability \cite{chetaev1961stability}. By \cite[Theorem 2.14]{braun2021stability}, an equilibrium point of ordinary differential equation is unstable if there exists a continuous function with positive values in any neighborhood of the equilibrium where it is equal to zero, and it is increasing along any trajectory (see also \cite[Theorems 5.29 and 5.30]{sastry2013nonlinear}). Chetaev functions have gained renewed interest recently in the context of obstacle avoidance in control, where one seeks to render the obstacles unstable by feedback \cite[Section III. B.]{braun2019uniting} \cite[Section IV. B.]{braun2018complete}. 
The function $C:\mathbb{R}^n \rightarrow \mathbb{R}$ in Theorem \ref{thm:suff_unstable} plays the role of a Chetaev function in a neighborhood of the point $x^*$. So long as the iterates $x_k$ stay near $x^*$ and avoid the critical points, the Chetaev function values $C(x_k)$ increase. If the increase is lower bounded by a positive constant at every iteration (i.e., when $\theta_1=0$), then we may readily conclude. Otherwise, the local geometry of the objective function comes into play. 

The fact that the exponent in the metric subregularity of the subdifferential is greater than one prevents the objective function from having a locally Lipschitz gradient if it is differentiable. The Verdier condition characterizes how fast subgradients become normal to the set of critical points in the vicinity of $x^*$. Together, these two conditions ensure that the iterates of the subgradient method do not converge to the set of critical points around $x^*$. Then the values $C(x_k)$ converge to plus infinity if the iterates remain near $x^*$, resulting in instability. In contrast, Bianchi \textit{et al.} \cite[Proposition 4]{bianchi2022stochastic} and Davis \textit{et al.} \cite[Proposition 5.2]{davis2023active} use the Verdier condition to ensure that the projection of the iterates on an active manifold containing a saddle point correspond to an inexact Riemannian gradient method with an implicit retraction. This technique is thus not suitable for proving instability of local minima.

In order to avoid assuming that the inequality $C(x_{k+1}) -  C(x_k) \geqslant c_1 d(x_k,S)^{\theta_1}$ holds for all $k \in \mathbb{N}$ in Theorem \ref{thm:suff_unstable}, one may require the Chetaev function to be convex and $\langle s , s' \rangle \leqslant  -d(x,S)^{\theta_1}$ for all $x\in U \setminus S, s \in \partial C(x),$ and $s' \in \partial f(x)$. Indeed, we then have $C(x_{k+1}) - C(x_k) \geqslant \langle s_k,x_{k+1}-x_k\rangle = \langle s_k,-\alpha s_k'\rangle \geqslant \alpha d(x,S)^{\theta_1}$ where $s_k \in \partial C(x_k)$ and $s_k' \in \partial f(x_k)$. These slightly stronger conditions hold in the first example in the next section.

\section{Applications}
\label{sec:Applications}

In this section, we apply Theorem \ref{thm:suff_unstable} to two examples using the Euclidean inner product. We first show that instability occurs in an example of ReLU neural network with $\ell_1$ loss, namely $(x_1,x_2,x_3)\in \mathbb{R}^3 \mapsto  |x_3 \max\{x_2,0\} - 1| + |x_3 \max\{x_1 + x_2,0\}|$. Indeed, it is the loss function when one seeks to fit the ReLU neural network $(a_1,a_2) \in \mathbb{R}^2 \mapsto x_3 \max\{x_1 a_1 + x_2 a_2,0\}$ over two data points $(0,1)$ and $(1,1)$ with corresponding labels $1$ and $0$. Figure \ref{fig:nn_relative} reveals that the iterates of the subgradient method move away from a fixed spurious local minimum despite being initialized nearby. Five trials are displayed, each corresponding to a uniform choice of constant step size in $[0.05,0.15]$ and a random initial point within $10^{-3}$ relative distance of the local minimum. Figure \ref{fig:nn_lyapunov} shows the corresponding values of an associated Chetaev function $C:\mathbb{R}^3\rightarrow\mathbb{R}$ defined by $C(x_1,x_2,x_3) := 1-x_1$. The fact that this function must increase indefinitely if the iterates remain near the local minimum is at the root of the instability (see Proposition \ref{prop:neural_network_unstable}). Figure \ref{fig:nn_fv} shows that the objective function values eventually stabilize around the global minimum value.

\begin{figure}[ht]
    \centering
     \begin{subfigure}{.49\textwidth}
  \centering
  \includegraphics[width=.95\textwidth]{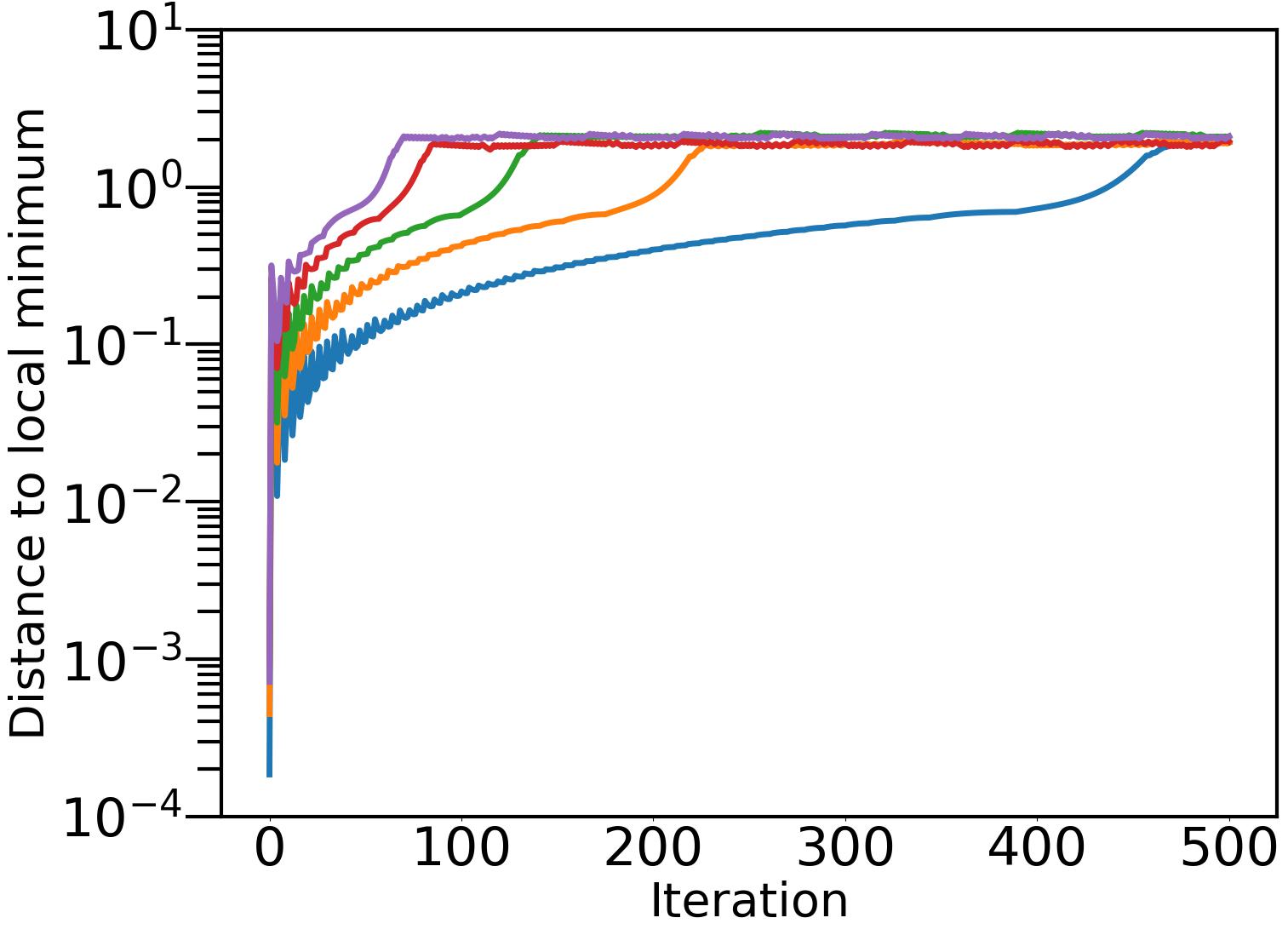}
  \caption{}
  \label{fig:nn_relative}
\vspace*{1mm}
  \label{fig:nn_distance}
 \end{subfigure}
 \begin{subfigure}{.49\textwidth}
  \centering
  \includegraphics[width=.95\textwidth]{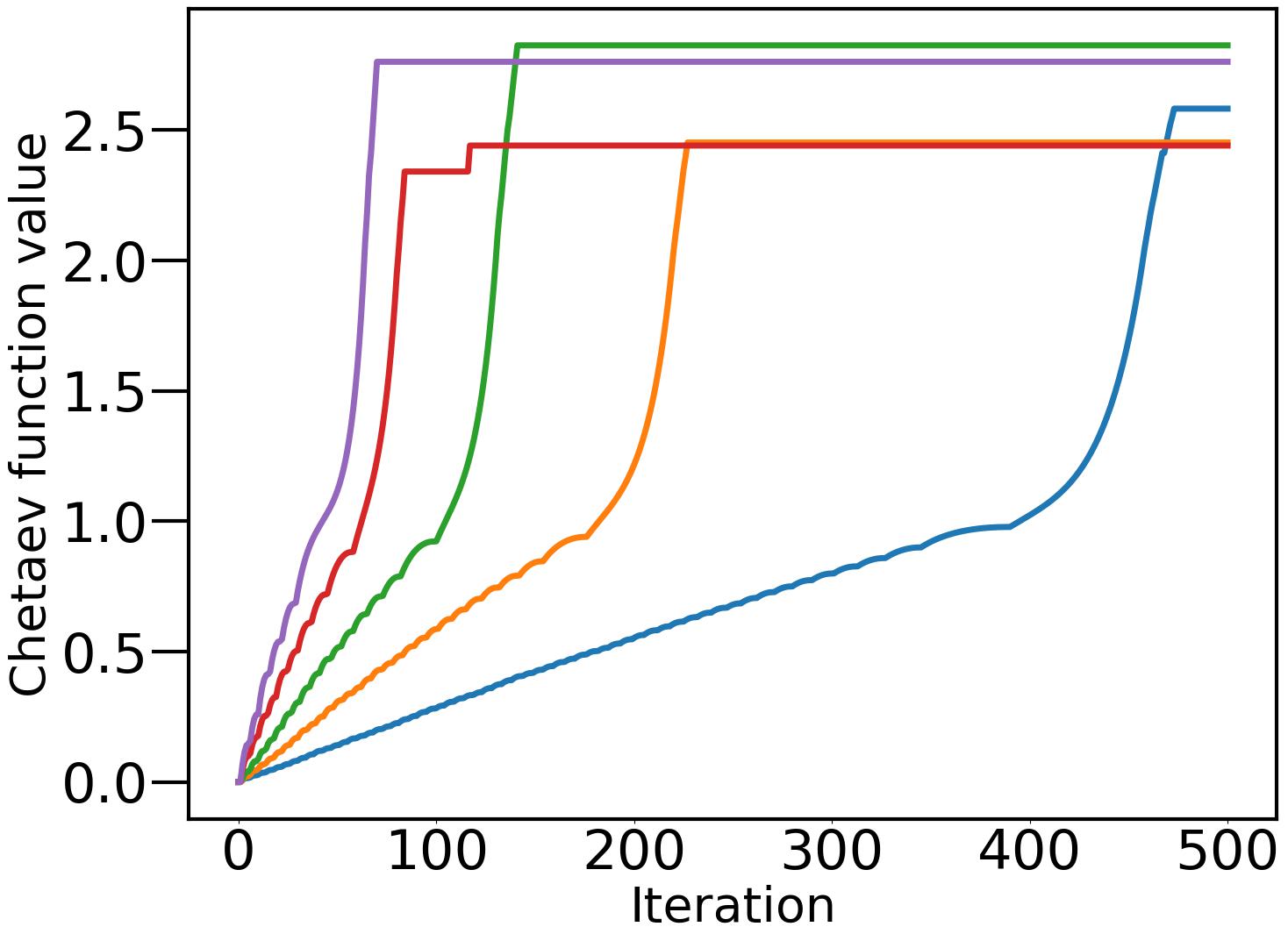}
  \caption{}
  \label{fig:nn_lyapunov}
 \end{subfigure}
  \begin{subfigure}{.49\textwidth}
  \centering
  \includegraphics[width=.95\textwidth]{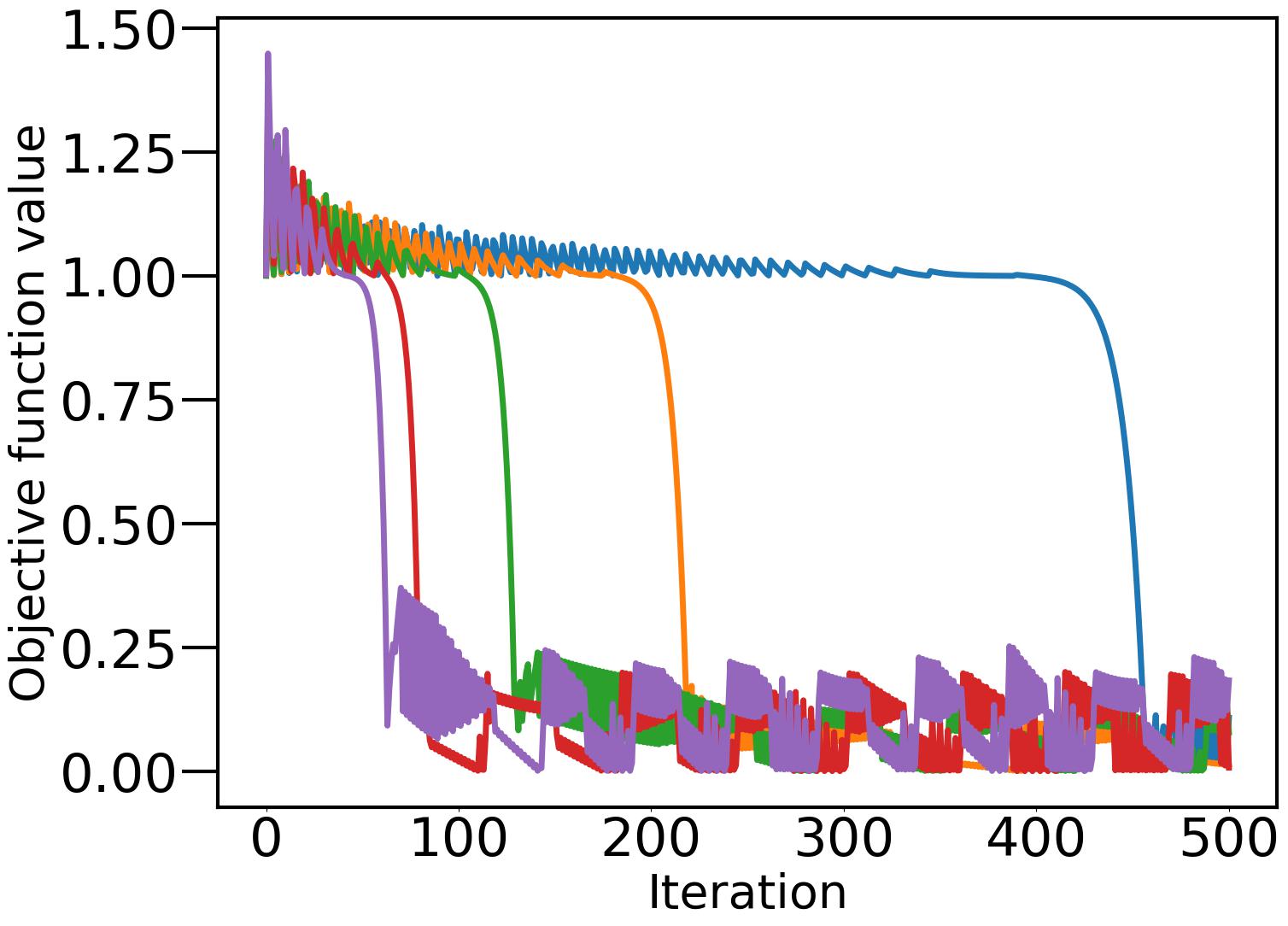}
  \caption{}
  \label{fig:nn_fv}
 \end{subfigure}
    \caption{Subgradient method randomly initialized near a spurious local minimum of a ReLU neural network with $\ell_1$ loss (5 trials with different step sizes).}
     \label{fig:nn}
\end{figure}

\begin{proposition}
\label{prop:neural_network_unstable}
The point $(1,1,0)$ is a strongly unstable spurious local minimum of the function defined from $\mathbb{R}^3$ to $\mathbb{R}$ by $f(x_1,x_2,x_3) :=  |x_3 \max\{x_2,0\} - 1| + |x_3 \max\{x_1 + x_2,0\}|$.
\end{proposition}
\begin{proof}
There exists a neighborhood $U$ of $(1,1,0)$ such that for all $(x_1,x_2,x_3)\in U$, we have $x_1\geqslant 1/2$, $x_2 \geqslant 1/2$ and $x_2 x_3 < 1$. Thus inside $U$ we have $f(x_1,x_2,x_3) = |x_3 x_2-1|+|x_3(x_1+x_2)| = 1-x_3 x_2 + |x_3|(x_1+x_2) = 1 + x_2(|x_3|-x_3)+x_1|x_3|\geqslant f(1,1,0) > f(-1,1,1)$, with equality in the inequality if and only if $x_3=0$. It follows that $(1,1,0)$ is a spurious local minimum. We next show that it is strongly unstable using Theorem \ref{thm:suff_unstable}. The function $f$ is locally Lipschitz and semi-algebraic. Let $S$ denote the set of critical points of $f$. By the definable Morse-Sard theorem \cite[Corollary 9]{bolte2007clarke} and by shrinking the neighborhood $U$ if necessary, $S \cap U = \{(x_1,x_2,x_3)\in U : x_3 = 0\}$ is a $C^2$ manifold of dimension $2$ at $(1,1,0)$. Let  $\theta_1 := 1$ and $C:\mathbb{R}^3\rightarrow\mathbb{R}$ be the continuous function defined by $C(x_1,x_2,x_3) := 1 - x_1$. Let $\alpha>0$ and consider a sequence $(x_1^k,x_2^k,x_3^k)_{k \in \mathbb{N}}$ generated by the subgradient method with constant step size $\alpha$ such that $(x_1^k,x_2^k,x_3^k) \in U\setminus S$ for all $k\in \mathbb{N}$. Letting $c_1 :=\alpha$, we have $C(x_1^{k+1},x_2^{k+1},x_3^{k+1}) - C(x_1^{k},x_2^{k},x_3^k) = x_1^k - x_1^{k+1} =  \alpha |x_3^k| = c_1 d((x_1^{k},x_2^{k},x_3^k),S)^{\theta_1}$ for all $k \in \mathbb{N}$. Letting $\theta_2 := 1$, and $c_2:= 1$,
we have $d((x_1,x_2,x_3),S) \leqslant c_2 d(0,\partial f(x_1,x_2,x_3))^{\theta_2}$ for all $(x_1,x_2,x_3) \in U$, so $\partial f$ is metrically $\theta_2$-subregular at $((1,1,0),(0,0,0))$. Finally, letting $c_3 := \sqrt{5}$, for all $(x_1,x_2,x_3) \in  U\setminus S$, $(y_1,y_2,y_3) \in S\cap U$, and $(v_1,v_2,v_3) \in \partial f((x_1,x_2,x_3))$, we have $\|P_{T_S(y_1,y_2,y_3)}(v_1,v_2,v_3) - \nabla_S f(y_1,y_2,y_3)\|^2 = \|(v_1,v_2,0)\|^2 = |x_3|^2 + (|x_3| - x_3)^2 \leqslant 5x_3^2 \leqslant c_3^2\|(x_1,x_2,x_3) - (y_1,y_2,0)\|^2$. Thus $f$ satisfies the Verdier condition at $(1,1,0)$ along $S$.
\end{proof}

Second, we show that instability occurs in an example of robust principal component analysis with real-world data. The objective function $f:\mathbb{R}^{m\times r} \times \mathbb{R}^{n\times r}\rightarrow\mathbb{R}$ is defined by $f(X,Y) := \|XY^T-M\|_1$ \cite[Equation (4)]{gillis2018complexity} where $\|\cdot\|_1$ is the entrywise $\ell_1$-norm of a matrix and $M \in \mathbb{R}^{m\times n}$ is a data matrix. The goal is to decompose $M$ as a low rank matrix plus a sparse matrix. Figure \ref{fig:rpca_relative} reveals that the iterates of the subgradient method move away from a fixed spurious local minimum $(X^*,Y^*)\in \mathbb{R}^{m\times r} \times \mathbb{R}^{n\times r}$ despite being initialized nearby ($(m,n,r)=(62400,3417,10)$ in the experiment). Five trials are displayed, each corresponding to a uniform choice of constant step size in $[0.0000025,0.0000075]$ and a random initial point within $10^{-3}$ relative distance of the local minimum. Figure \ref{fig:rpca_lyapunov} shows the corresponding values of an associated Chetaev function $C:\mathbb{R}^{m\times r} \times \mathbb{R}^{n\times r}\rightarrow\mathbb{R}$ defined by $C(X,Y) := \|X^*\|_F^2-\|Y^*\|_F^2+\|Y\|_F^2-\|X\|_F^2$ where $\|\cdot\|_F$ is the Frobenius norm. The function increases as long as the iterates remain near the local minimum, but ceases to do so once the iterates are far enough. This is sufficient to prove instability (see Proposition \ref{prop:rpca_rankr_unstable}). Figure \ref{fig:rpca_fv} shows that the objective function values eventually drop below the spurious critical value and stabilize around a new value. 

\begin{figure}[ht]
    \centering
     \begin{subfigure}{.49\textwidth}
  \centering
  \includegraphics[width=.95\textwidth]{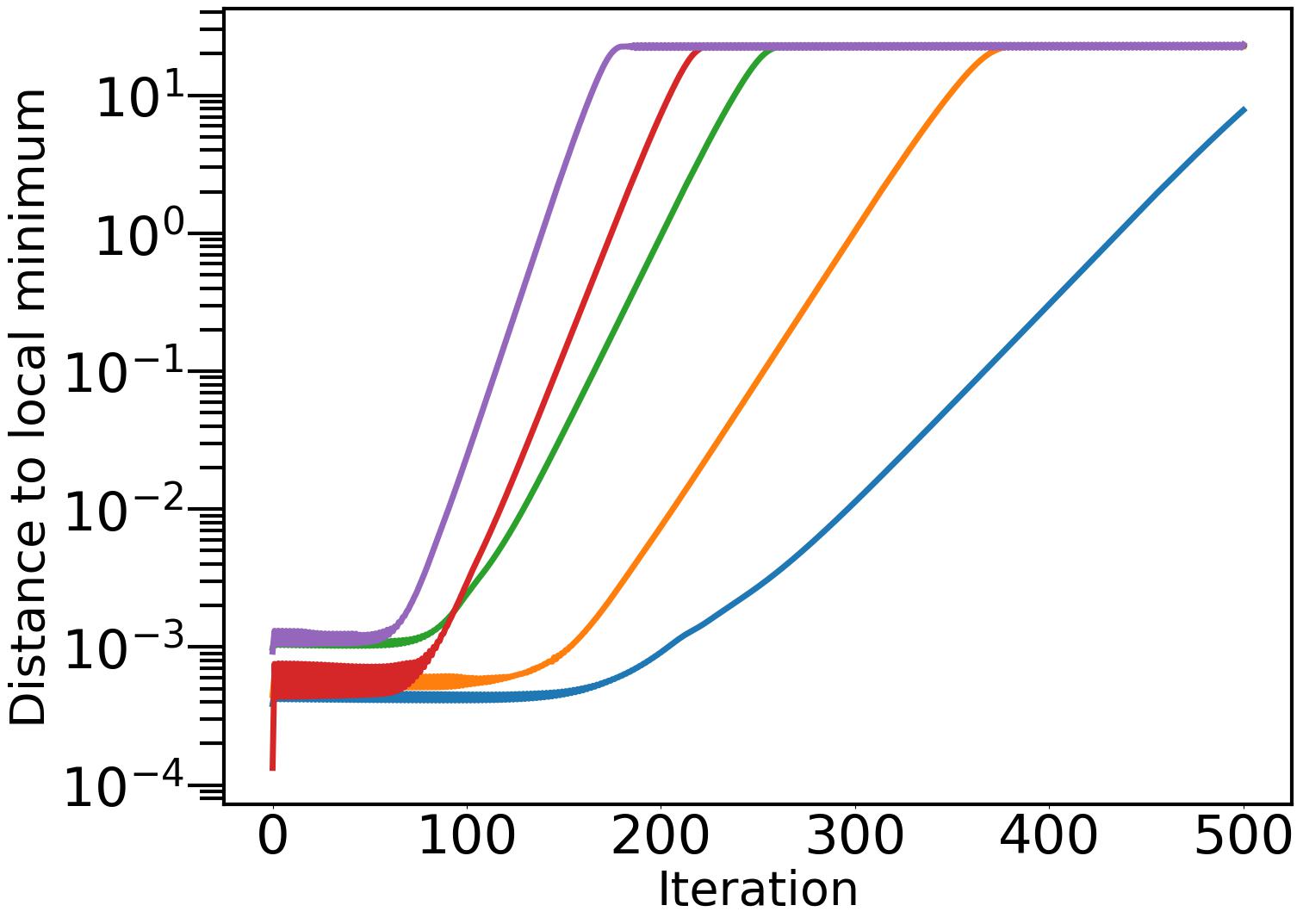}
  \caption{}
  \label{fig:rpca_relative}
\vspace*{1mm}
  \label{fig:video_distance}
 \end{subfigure}
 \begin{subfigure}{.49\textwidth}
  \centering
  \includegraphics[width=.95\textwidth]{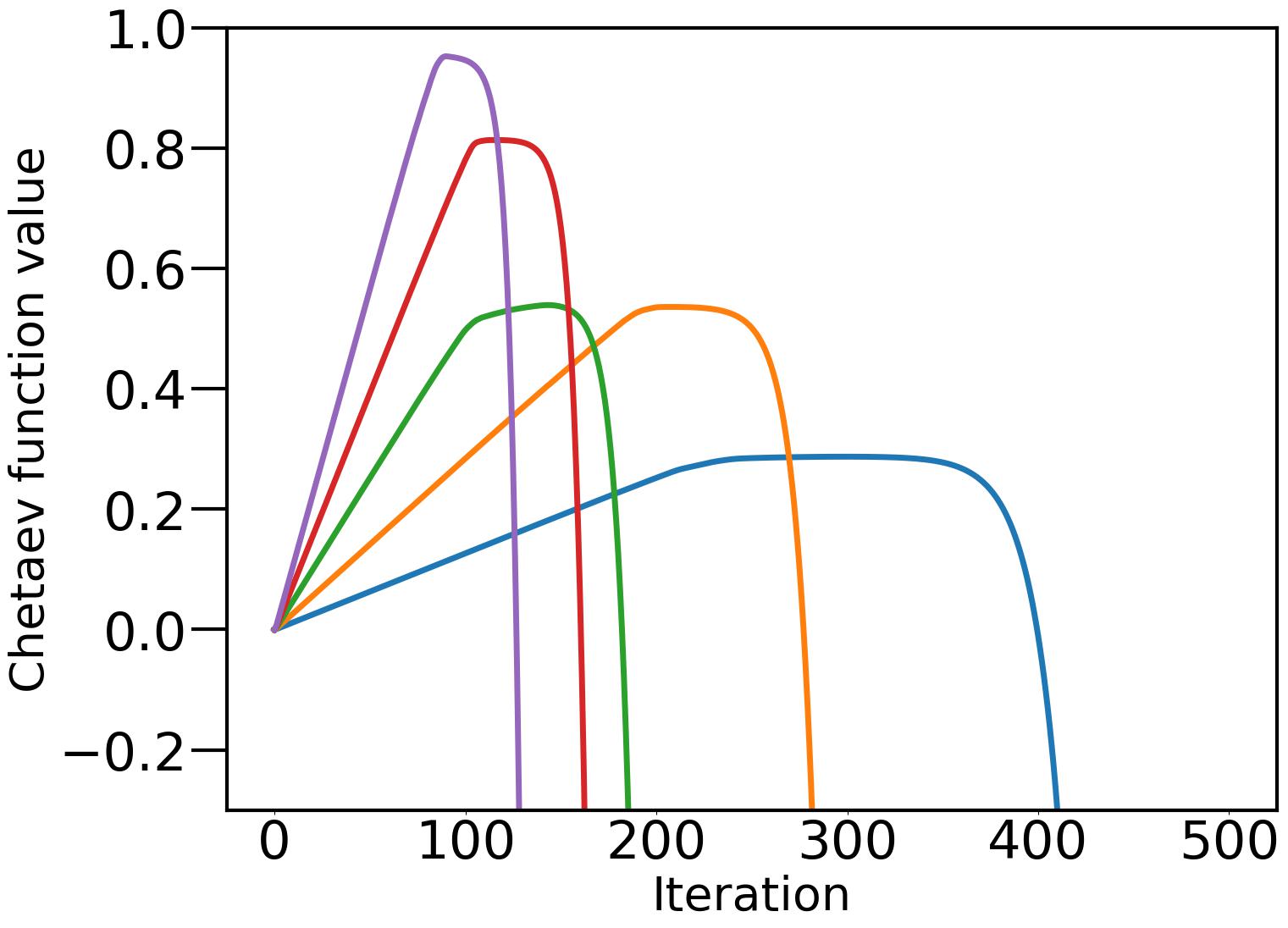}
  \caption{}
  \label{fig:rpca_lyapunov}
 \end{subfigure}
 \begin{subfigure}{.49\textwidth}
  \centering
  \includegraphics[width=.95\textwidth]{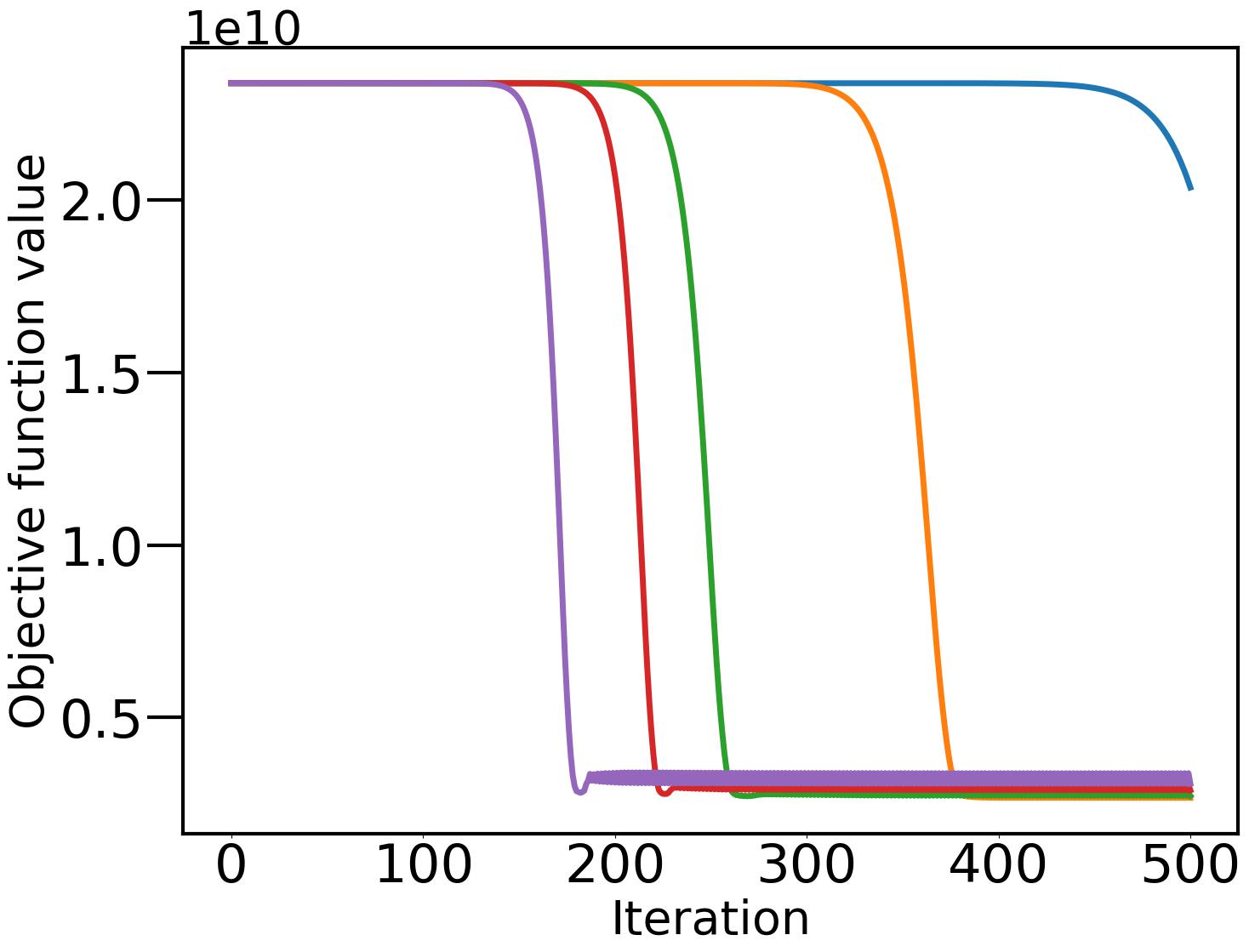}
  \caption{}
  \label{fig:rpca_fv}
 \end{subfigure}
    \caption{Subgradient method randomly initialized near a spurious local minimum of robust principal component analysis (5 trials with different step sizes).}
     \label{fig:rpca}
\end{figure}

The data used in Figure \ref{fig:rpca} is used in \cite[Figure 3]{netrapalli2014} to illustrate \textit{Non-convex Alternating Projections based Robust PCA} \cite[Algorithm 1]{netrapalli2014} and comes from the same dataset as the one used to illustrate \textit{Principal Component Pursuit} \cite[Equation (1.1)]{candes2011}. The application in those works consists of detecting moving objects in a surveillance video. Spurious local minima exist because the data matrix has zero rows, which corresponds to pixels that are composed of at most two of the three primary colors (red, green, and blue) throughout the video. It is crucial that the iterates of the subgradient method do not remain near a spurious local minimum like the one in Figure \ref{fig:rpca}. Otherwise, no moving object would be detected. In contrast, at the lower value obtained in Figure \ref{fig:rpca_fv}, all moving objects are detected. This can be seen in Figure \ref{fig:video_frames} and at the link \href{https://www.youtube.com/playlist?list=PLIR8kg8LvAmuPPZPD4CxjoqcYskoUzjzN}{[video]}.

 \begin{figure}[ht!]

 \centering
 \begin{subfigure}{.327\textwidth}
  \centering
  \includegraphics[width=1\textwidth]{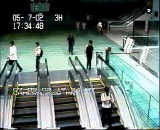}
  \caption{Original frame.}
  \label{fig:original_140}
 \end{subfigure}
 \begin{subfigure}{.327\textwidth}
  \centering
  \includegraphics[width=1\textwidth]{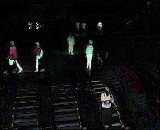}
  \caption{Moving objects.}
  \label{fig:s_140}
 \end{subfigure}
 \begin{subfigure}{.327\textwidth}
  \centering
  \includegraphics[width=1\textwidth]{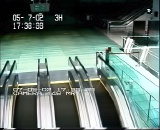}
  \caption{Background.}
  \label{fig:l_140}
 \end{subfigure}
 \vspace*{-1mm}
 \caption{Output after 500 iterations of the subgradient method with step size $\alpha = 0.000005$ when initialized within $10^{-3}$ of a spurious local minimum in relative distance.}
 \label{fig:video_frames}
 \vspace*{-3mm}
 \end{figure}
\begin{proposition}
\label{prop:rpca_rankr_unstable}
	The function $f$ defined from $\mathbb{R}^{m\times r} \times \mathbb{R}^{n\times r}$ to $\mathbb{R}$ by $f(X,Y):= \|XY^T - M\|_1$ admits strongly unstable spurious local minima if $M \in \mathbb{R}^{m\times n} \setminus \{0\}$ and $M$ contains at least $r$ zero rows or $r$ zero columns.
\end{proposition}
\begin{proof}
Without loss of generality, we assume that the first $r$ rows of $M$ are equal to zero. Let $\tilde{M}$ be the matrix containing the $m-r$ remaining rows, one of which is non-zero. We seek to show that $(X^*,Y^*) \in \mathbb{R}^{m\times r} \times \mathbb{R}^{n \times r}$ is a strongly unstable spurious local minimum of $f(X,Y):=\|XY^T-M\|_1$, where the first $r$ rows of $X^*$ form an invertible matrix $\hat{X}^* \in \mathbb{R}^{r \times r}$, the remaining rows are zero, and $Y^* = 0$. Given $(H,K) \in \mathbb{R}^{m\times r} \times \mathbb{R}^{n \times r}$, let $\hat{H}$ be the first $r$ rows of $H$ and let $\tilde{H}$ be the remaining $m-r$ rows. For all $(H,K)$ sufficiently small, we have
\begin{subequations}
\label{eq:hk}
    \begin{align}
        f(X^*+H,Y^*+K) & = \| (X^*+H)K^T - M \|_1  \\
         & =   \|(\hat{X}^* + \hat{H})K^T\|_1 + \|\tilde{H}K^T - \tilde{M}\|_1 \label{split} \\
         & \geqslant  \|\hat{X}^*K^T\|_1 - \|\hat{H}K^T\|_1 + \|\tilde{M}\|_1 - \|\tilde{H}K^T\|_1 \label{triangle}\\
         & = \|\hat{X}^*K^T\|_1 ~-~ \|HK^T\|_1 ~+~ \|M\|_1 \label{group}\\
         & \geqslant c\|K\|_1 ~-~ \|HK^T\|_1 ~+~ \|M\|_1 \label{equivalence_norms} \\
         & \geqslant c/2\|K\|_1 ~+~ \|M\|_1 \label{dual_norm} \\
         & \geqslant \|M\|_1 = f(X^*,Y^*) > f(\bar{X},\bar{Y}). \label{H_small}
    \end{align}
\end{subequations}
Above, the first term in \eqref{split} is the $\ell_1$-norm of the first $r$ rows of $(X^*+H)K^T - M$, while the second term in \eqref{split} is the $\ell_1$-norm of the remaining rows. \eqref{triangle} follows from the triangular inequality. \eqref{group} holds because the first $r$ rows of $HK^T$ are $\hat{H}K^T$, while the remaining rows are $\tilde{H}K^T$. The existence of a positive constant $c$ in \eqref{equivalence_norms} is due to the equivalence of norms ($K\mapsto \|\hat{X}^*K^T\|_1$ is a norm because $\hat{X}^*$ is invertible). \eqref{dual_norm} holds because we may take $\|H\|_\infty \leqslant c/(2m)$ where $\|\cdot\|_\infty$ is the dual norm of $\|\cdot\|_1$. Then $\|HK^T\|_1 = \sum_{i=1}^m \sum_{j=1}^n | \langle h_i , k_j \rangle | \leqslant \sum_{i=1}^m \sum_{j=1}^n \|h_i\|_\infty \|k_j\|_1 \leqslant \sum_{i=1}^m \sum_{j=1}^n \|H\|_\infty \|k_j\|_1 \leqslant c/2 \sum_{j=1}^n \|k_j\|_1 = c/2 \|K\|_1$ where $h_i^T$ and $k_j^T$ respectively denote the rows of $H$ and $K$. Finally, we may choose $(\bar{X},\bar{Y})$ in \eqref{H_small} to be factors of a rank-one matrix $\bar{M} \in \mathbb{R}^{m\times n}$ which has all zero entries, apart from one where $\bar{M}_{ij} = M_{ij} \neq 0$. Then $f(\bar{X},\bar{Y}) = \|\bar{X}\bar{Y}^T-M\|_1 = \|\bar{M}-M\|_1 = \|M\|_1 - |M_{ij}| < \|M\|_1$.

We next show that $(X^*,Y^*)$ is strongly unstable using Theorem \ref{thm:suff_unstable}. The function $f$ is locally Lipschitz and semi-algebraic. Let $S$ denote the set of critical points of $f$. By the definable Morse-Sard theorem \cite[Corollary 9]{bolte2007clarke}, there exists a bounded neighborhood $U$ of the local minimum $(X^*,Y^*)$ such that $S \cap U = \{(X,Y) \in U: f(X,Y) = f(X^*,Y^*)\} = \{(X,Y) \in U:Y = Y^*\}$, where the second setwise equality is due to \eqref{dual_norm}. As a result, $S$ is a $C^2$ manifold at $(X^*,Y^*)$. Let $\theta_1 := 0$ and $C:\mathbb{R}^{m \times r}\times \mathbb{R}^{n \times r}\rightarrow \mathbb{R}$ be the continuous function defined by $C(X,Y):= \|X^*\|_F^2-\|Y^*\|_F^2+\|Y\|_F^2-\|X\|_F^2$. Let $\alpha>0$ and consider a sequence $(X_k,Y_k)_{k \in \mathbb{N}}$ generated by the subgradient method with constant step size $\alpha$ such that $(X_k,Y_k) \in U\setminus S$ for all $k\in \mathbb{N}$. Let $\mathrm{sign}(\cdot)$ be the function defined by $\mathrm{sign}(t) = 1$ if $t>0$, $\mathrm{sign}(t) = -1$ if $t<0$, and $\mathrm{sign}(t) = [-1,1]$ if $t=0$. When the input is a matrix, it is applied entrywise. Letting 
\begin{equation*}
    c_1 :=\alpha^2 \inf \{ \|\Lambda^T X\|_F^2 - \|\Lambda Y\|_F^2 : (X,Y) \in U\setminus S, ~ \Lambda \in \mathrm{sign}(X Y^T - M) \},
\end{equation*}
we have $C(X_{k+1},Y_{k+1}) - C(X_k,Y_k) = \hdots$
\begin{subequations}
\begin{align}
    & = \|Y_{k+1}\|_F^2 - \|X_{k+1}\|_F^2 - \|Y_{k}\|_F^2 + \|X_{k}\|_F^2 \label{eq:chpca_b}\\
    & = \mathrm{trace}(Y_{k+1}^T Y_{k+1} - X_{k+1}^T X_{k+1} - Y_k^T Y_k + X_k^T X_k) \label{eq:chpca_c}\\
    & = \alpha^2 \mathrm{trace}(X_k^T \Lambda_k \Lambda_k^T X_k - Y_k^T \Lambda_k^T\Lambda_k Y_k) \\
    & = \alpha^2 (\|\Lambda_k^T X\|_F^2 - \|\Lambda_k Y_k\|_F^2)  \geqslant c_1 d((X_k,Y_k),S)^{\theta_1} \label{eq:chpca_d}
\end{align}
\end{subequations}
where $\Lambda_k \in \mathrm{sign}(X_k Y_k^T - M)$. It remains to show that $c_1>0$. Given $(\Lambda,X,M) \in \mathbb{R}^{m\times n} \times \mathbb{R}^{m\times r}\times \mathbb{R}^{m\times n}$, let $(\hat{\Lambda},\hat{X},\hat{M})$ be the first $r$ rows of $(\Lambda,X,M)$ and let $(\tilde{\Lambda},\tilde{X},\tilde{M})$ be the remaining $m-r$ rows. It suffices to show that
\begin{equation*}
\label{eq:lambdah}
    \inf \{ \|\hat{\Lambda}^T \hat{X}\|_F : (X,Y) \in U \setminus S,~\hat{\Lambda} \in \mathrm{sign}(\hat{X}Y^T - \hat{M}) \} > 0
\end{equation*}
after possibly reducing the neighborhood $U$ of $(X^*,Y^*)$. Indeed, for all $(X,Y) \in U \setminus S$ and $\Lambda \in \mathrm{sign}(X Y^T - M) $, we then have $\|\Lambda^T X\|_F^2 - \|\Lambda Y\|_F^2 = \|\hat{\Lambda}^T \hat{X} + \tilde{\Lambda}^T \tilde{X}\|_F^2 - \|\Lambda Y\|_F^2 \geqslant (\|\hat{\Lambda}^T \hat{X}\|_F - \|\tilde{\Lambda}^T \tilde{X}\|_F)^2 - \|\Lambda Y\|_F^2 \geqslant \|\hat{\Lambda}^T \hat{X}\|_F^2/2$ since $\tilde{X}^* = Y^* = 0$. We next reason by contradiction and assume that the infimum in \eqref{eq:lambdah} is equal to zero.
Let $(X_i,Y_i,\Lambda_i)_{i \in \mathbb{N}}$ be a minimizing sequence. Since it is contained in the bounded set $U \times [-1,1]^{m \times n}$, there exists a subsequence (again denoted $(X_i,Y_i,\Lambda_i)_{i \in \mathbb{N}}$) that converges to some $(X^\circ,Y^\circ,\Lambda^\circ)$. Naturally we have $(\hat{\Lambda}^\circ)^T \hat{X}^\circ = 0$. On the one hand, since $\hat{X}^* \in \mathbb{R}^{r\times r}$ is invertible, so is any matrix in its neighborhood $\bar{U}$, in particular $\hat{X}^\circ,\hat{X}_0,\hat{X}_1,\hdots$ after possibly reducing $U$. Hence $\hat{\Lambda}^\circ = 0$. On the other hand, since $(X_i,Y_i) \in U \setminus S$, $S \cap U = \{ (X,Y)\in U : Y=Y^*=0\}$, and $\hat{M} = 0$, we have $Y_i \neq 0$ and $\hat{X}_iY_i^T-\hat{M} \neq 0$ for all $i \in \mathbb{N}$. Hence the matrix $\hat{\Lambda}_i \in \mathrm{sign}(\hat{X}Y^T - \hat{M})$ has at least one entry equal to either $1$ or $-1$. Thus $\|\hat{\Lambda}_i\|_\infty \geqslant 1$ for all $i \in \mathbb{N}$. Passing to the limit, we obtain the contradiction $0 =\|\hat{\Lambda}^\circ\|_\infty \geqslant 1$.
\end{proof}
\section*{Acknowledgments}
We thank the reviewers and the associate editor for their valuable feedback.
\bibliographystyle{abbrv}    
\bibliography{mybib}
\end{document}